\documentclass[12pt]{amsart}

\makeatletter
\newcommand\mathcircled[1]{%
	\mathpalette\@mathcircled{#1}%
}
\newcommand\@mathcircled[2]{%
	\tikz[baseline=(math.base)] \node[draw,ellipse,inner sep=1pt] (math) {$\m@th#1#2$};%
}

\usepackage{multicol}
\usepackage{latexsym}
\usepackage{psfrag}
\usepackage{amsmath}
\usepackage{amssymb}
\usepackage{epsfig}
\usepackage{amsfonts}
\usepackage{amscd}
\usepackage{mathrsfs}
\usepackage{graphicx}
\usepackage{enumerate}
\usepackage[autostyle=false, style=english]{csquotes}
\MakeOuterQuote{"}
\usepackage{ragged2e}
\usepackage[all]{xy}
\usepackage{mathtools}
\usepackage{comment}

\theoremstyle{remark}
\newtheorem{remark}{Remark}

\newlength\ubwidth

\usepackage[dvipsnames]{xcolor}

\usepackage{placeins}
\usepackage{tikz}
\usetikzlibrary{shapes, positioning}
\usetikzlibrary{matrix,shapes.geometric,calc,backgrounds}

\newtheorem{theorem}{Theorem}

\newtheorem{ex}[theorem]{Example}

\newtheorem{cor}[theorem]{Corollary}

\usepackage{amsmath,amssymb,amsthm}
\usepackage[dvipsnames]{xcolor}
\usepackage[colorlinks=true,citecolor=RoyalBlue,linkcolor=red,breaklinks=true]{hyperref}

\usepackage{graphicx}
\usepackage{mathdots} 
\usepackage[margin=2.5cm]{geometry}
\usepackage{ytableau}

\usepackage{cancel}
\usepackage{blkarray}
\usepackage{booktabs}
\usepackage{microtype}

\begin{document}

	\title[]
	{On Partitions with Palindromic Gap Sequences}

	\author[\"{O}mr\"{u}uzun Seyrek]{Hal\.{ı}me \"{O}mr\"{u}uzun Seyrek}
	\address{Hal\.{ı}me \"{O}mr\"{u}uzun Seyrek, Faculty of Applied Sciences, 
		Biruni University, Zeytinburnu, Istanbul, 34015, Turkey}
	\email{hseyrek@biruni.edu.tr}
	
	
	
	\date{2026}
	
\begin{abstract}
We introduce and study \emph{palindromic-gap partitions}: partitions whose sequence of successive differences between consecutive parts is a palindrome. We derive the generating function for this family, treating partitions with an odd and an even number of parts separately, and obtain a bivariate refinement tracking the number of parts. Specializing further, we give explicit closed formulas for the number of palindromic-gap partitions of $n$ into exactly $r$ parts. We then construct explicit bijections between palindromic-gap partitions with $2m+1$, respectively $2m$, parts and ordinary partitions with at most $m+1$ parts, in which the weight of a palindromic-gap partition is encoded by the largest part, or the two largest parts, of its image; composing these yields an explicit bijection between the odd and even families themselves. Finally, we show that a partition is palindromic-gap if and only if its Ferrers diagram, drawn with left-justified rows, is self-complementary under a $180^\circ$ rotation inside its own naturally associated rectangle, identifying palindromic-gap partitions with self-complementary partitions relative to this rectangle. This yields a partial answer to a question of Keith on the enumeration of complementable and self-complementary partitions: we give explicit generating functions for the self-complementary case and exhibit, at $n=15$, a partition that is complementable but not self-complementary, confirming that the two notions are genuinely distinct.
\end{abstract}

\maketitle

	\section{Introduction}
	\label{secIntro}

An (ordinary) integer partition $\lambda$ is an expression of a 
non-negative integer by an unordered sum of positive integers~\cite{Andrews1998}.  
\begin{align*}
	n = \lambda_1 + \lambda_2 + \cdots + \lambda_m
\end{align*}
The parts $\lambda_i$'s may be written 
in non-increasing or non-decreasing order, 
and a limited number of zeros may be allowed 
to appear as parts.  
The number $n$ being partitioned is called the \emph{weight}, 
denoted by $\left\vert \lambda \right\vert$
and the number of parts $m$ is called the \emph{length} of the partition, denoted by $\ell(\lambda)$.  
For example, 4 has the following five partitions. 
\begin{align*}
	4, \quad
	1+3, \quad 
	2+2, \quad 
	1+1+2, \quad 
	1+1+1+1.
\end{align*}

Throughout this paper, we write partitions by separating their parts with commas rather than plus signs. Thus, for example, the partition $1, 2, 5, 5, 7$ denotes the partition $1+2+5+5+7$. 

Partitions defined by restrictions on the gap sequence between consecutive parts form one of the most extensively studied classes in partition theory. They go back to the Rogers–Ramanujan identities and their vast generalizations by Gordon \cite{Gordon1961} and by Andrews and Gordon \cite{AndrewsGordon1974}, where the difference between consecutive parts is the defining constraint. Other works studying restrictions on the differences between 
parts of a partition include those of
Schur~\cite{Schur1926},
Bressoud~\cite{Bressoud1980},
Andrews~\cite{Andrews1971},
Yee~\cite{Yee2004},
Andrews, Beck, and Robbins~\cite{AndrewsBeckRobbins2015},
and Bogachev and Yakubovich~\cite{BogachevYakubovich2019}. In this paper we study a natural variant of this theme: partitions whose gap sequence is required to be \textit{symmetric}.	

Recently, Hemmer and Westrem~\cite{HemmerWestrem2024} studied \textit{palindrome partitions.} Although we use similar terminology, the partitions we study are unrelated to palindrome partitions. We require
the sequence of successive differences between consecutive parts to be
palindromic. To the best of our knowledge, palindromic-gap partitions, as defined here, have not been previously studied.

We say the partition  $\lambda=\lambda_1, \lambda_2, \ldots, \lambda_m$ has a palindromic gap sequence if the sequence of consecutive differences  

\begin{align*}
g(\lambda)=(\lambda_2-\lambda_1, \lambda_3-\lambda_2, \ldots, \lambda_m-\lambda_{m-1})
\end{align*}

is a palindrome, that is

\begin{align*}
\lambda_{i+1}-\lambda_i=\lambda_{m-i+1}-\lambda_{m-i}, \quad \text{$1\leq i \leq m-1$}.
\end{align*}

$\mathcal{P}_{pal}$ denotes the set of all partitions with this property. From now on, an element of $\mathcal{P}_{pal}$ will simply be called a palindromic-gap partition. The empty partition and all single-part partitions satisfy the condition vacuously. For example,  the partition $1, 3, 7, 9$ is a palindromic-gap partition, since the consecutive differences sequence $2, 4, 2$ is a palindrome. 

The remainder of the paper is organized as follows.

In Section~\ref{GF of partitions}, we construct the generating function for
$\mathcal{P}_{\mathrm{pal}}$, treating partitions with an odd and an even
number of parts separately, since the two cases require different
parametrizations of the palindromic gap sequence. Combining the two cases
(Theorem~\ref{GF}) yields a closed-form generating function for
$\mathcal{P}_{\mathrm{pal}}$, and we observe that it coincides with the
generating function of self-complementary ordinary partitions inside their
own naturally associated rectangles, though obtained here by a different,
partition-centered method. As a corollary, we refine this generating function
to a bivariate form that also keeps track of the number of parts.

In Section~\ref{enumeration}, we specialize the results of
Section~\ref{GF of partitions} to obtain, in
Corollaries~\ref{oddnumberofparts} and~\ref{evennumberofparts}, explicit
closed formulas --- again split according to the parity of the number of
parts --- for the number $p_{\mathrm{pg}}(n,r)$ of palindromic-gap partitions
of $n$ into exactly $r$ parts. Worked examples illustrate both formulas.

In Section~\ref{bijections}, we construct explicit bijections $\Phi_{2m+1}$
and $\Phi_{2m}$ (Theorems~\ref{thm:odd_bijection}
and~\ref{thm:even_bijection}) between palindromic-gap partitions and
ordinary partitions with at most $m+1$ parts, again treating the two
parities separately. These bijections are not weight-preserving in the
usual sense: rather than sending a partition of $n$ to a partition of $n$,
the weight of a palindromic-gap partition is encoded by the largest part of
its image, in the odd case, or by the two largest parts of its image, in
the even case. We describe the maps and their inverses explicitly, and
illustrate each with a worked example.

Finally, in Section~\ref{ferrers diagrams}, we give a geometric
characterization of palindromic-gap partitions. We show
(Theorem~\ref{thm:ferrers}) that $\lambda$ is a palindromic-gap partition if
and only if $\lambda_i+\lambda_{\ell+1-i}=\lambda_1+\lambda_\ell$ for all
$i$, and if and only if the Ferrers diagram of $\lambda$, drawn with
left-justified rows, is self-complementary under a $180^\circ$ rotation
inside the rectangle of dimensions $\ell\times(\lambda_1+\lambda_\ell)$.
This identifies the palindromic-gap partitions of $n$ precisely with the
ordinary partitions of $n$ that are self-complementary inside their own
natural bounding rectangle, so that the results of the two preceding
sections give explicit generating functions for the latter family as well.
As an immediate corollary, every palindromic-gap partition satisfies the
area condition $\ell(\lambda_1+\lambda_\ell)=2n$ that Keith~\cite{Keith2011}
uses to define a partition as \emph{complementable}; we then exhibit, at the
smallest possible value $n=15$, an explicit partition that is complementable
but not self-complementary, confirming with a concrete instance an
assertion that Keith records without proof or example. Taken together,
these results give a partial answer to a question posed by
Keith~\cite{Keith2011}, who asks for the enumeration, and for explicit
generating functions, of complementable and self-complementary partitions,
both ordinary and signed. We answer this question for ordinary partitions, and show that
complementability is a strictly weaker condition than self-complementarity.
The case of signed partitions, and the enumeration of complementable
partitions that fail to be self-complementary, remain open, as they did in
Keith's original question; we hope the techniques developed here will be
useful in resolving them.

\section{Generating Function of Partitions with Palindromic Gap Sequences}
\label{GF of partitions}
In this section, we construct the generating function of the palindromic-gap partitions. To distinguish palindromic-gap partitions according to the parity of their number of parts, we introduce the following subsets.
\[
\mathcal{P}^{\mathrm{odd}}_{\mathrm{pal}}
=
\{\lambda\in\mathcal{P}_{\mathrm{pal}}:\ell(\lambda)\text{ is odd}\},
\]

\[
\mathcal{P}^{\mathrm{even}}_{\mathrm{pal}}
=
\{\lambda\in\mathcal{P}_{\mathrm{pal}}:\ell(\lambda)\text{ is even}\}.
\]

\begin{theorem} \label{GF}
	The generating functions for palindromic-gap partitions with odd and even numbers of parts are respectively
	\[
	\sum_{\lambda\in\mathcal{P}^{\mathrm{odd}}_{\mathrm{pal}}}q^{|\lambda|}
	=
	\sum_{m\ge0}
	\frac{q^{2m+1}}
	{(1-q^{2m+1})^{m+1}},
	\]
	and
	\[
	\sum_{\lambda\in\mathcal{P}^{\mathrm{even}}_{\mathrm{pal}}}q^{|\lambda|}
	=
	\sum_{m\ge1}
	\frac{q^{2m}}
	{(1-q^{2m})^{m}(1-q^m)}.
	\]
	Consequently,
\begin{equation}	\label{ref:GF}
	\sum_{\lambda\in\mathcal{P}_{\mathrm{pal}}}q^{|\lambda|}
	=
	1+
	\sum_{m\ge0}
	\frac{q^{2m+1}}
	{(1-q^{2m+1})^{m+1}}
	+
	\sum_{m\ge1}
	\frac{q^{2m}}
	{(1-q^{2m})^{m}(1-q^m)}.
\end{equation}
\end{theorem}

\begin{proof}
	Let
	\[
	\lambda=(\lambda_1,\lambda_2,\ldots,\lambda_\ell)
	\]
	be a partition written in non-decreasing order, and define its successive gaps by
	\[
	d_i=\lambda_{i+1}-\lambda_i,
	\qquad 1\leq i\leq \ell-1.
	\]
	Thus, each \(d_i\) is a nonnegative integer. The partition \(\lambda\) is
	palindromic-gap precisely when
	\[
	d_i=d_{\ell-i},
	\qquad 1\leq i\leq \ell-1.
	\]
	
	We first consider partitions with an odd number \(2m+1\) of parts. Such a
	partition is uniquely determined by its smallest part
	\[
	a=\lambda_1\geq 1
	\]
	and the \(m\) independent gaps
	\[
	d_1,d_2,\ldots,d_m\geq 0,
	\]
	since the remaining gaps satisfy
	\[
	d_{2m+1-i}=d_i,
	\qquad 1\leq i\leq m.
	\]
	Since
	\[
	\lambda_j
	=
	a+\sum_{i=1}^{j-1}d_i,
	\]
	we have
	\[
	|\lambda|
	=
	(2m+1)a+\sum_{i=1}^{2m}(2m+1-i)d_i.
	\]
	Pairing the terms corresponding to \(d_i\) and \(d_{2m+1-i}=d_i\), we obtain
	\[
	(2m+1-i)d_i+i d_i=(2m+1)d_i.
	\]
	Therefore,
	\[
	|\lambda|
	=
	(2m+1)\left(a+d_1+\cdots+d_m\right).
	\]
	Consequently, the generating function for palindromic-gap partitions with
	\(2m+1\) parts is
	\[
	\sum_{\substack{\lambda\in\mathcal{P}_{\mathrm{pal}}\\
			\ell(\lambda)=2m+1}}
	q^{|\lambda|}
	=
	\left(\sum_{a\geq 1}q^{(2m+1)a}\right)
	\prod_{i=1}^{m}
	\left(\sum_{d_i\geq 0}q^{(2m+1)d_i}\right).
	\]
	Hence,
	\[
	\sum_{\substack{\lambda\in\mathcal{P}_{\mathrm{pal}}\\
			\ell(\lambda)=2m+1}}
	q^{|\lambda|}
	=
	\frac{q^{2m+1}}{(1-q^{2m+1})^{m+1}}.
	\]
	Summing over \(m\geq 0\) gives
	\[
	\sum_{\lambda\in\mathcal{P}_{\mathrm{pal}}^{\mathrm{odd}}}
	q^{|\lambda|}
	=
	\sum_{m\geq 0}
	\frac{q^{2m+1}}{(1-q^{2m+1})^{m+1}}.
	\]
	
	We now consider partitions with an even number \(2m\) of parts, where
	\(m\geq 1\). Such a partition is uniquely determined by
	\[
	a=\lambda_1\geq 1
	\]
	and the \(m\) independent gaps
	\[
	d_1,d_2,\ldots,d_m\geq 0.
	\]
	Here \(d_m\) is the central gap, while
	\[
	d_{2m-i}=d_i,
	\qquad 1\leq i\leq m-1.
	\]
	As above,
	\[
	|\lambda|
	=
	2ma+\sum_{i=1}^{2m-1}(2m-i)d_i.
	\]
	For \(1\leq i\leq m-1\), the paired gaps \(d_i\) and \(d_{2m-i}=d_i\)
	contribute
	\[
	(2m-i)d_i+i d_i=2m d_i.
	\]
	The central gap \(d_m\) contributes \(m d_m\). Thus,
	\[
	|\lambda|
	=
	2m\left(a+d_1+\cdots+d_{m-1}\right)+m d_m.
	\]
	It follows that
	\[
	\begin{aligned}
		\sum_{\substack{\lambda\in\mathcal{P}_{\mathrm{pal}}\\
				\ell(\lambda)=2m}}
		q^{|\lambda|}
		&=
		\left(\sum_{a\geq 1}q^{2ma}\right)
		\prod_{i=1}^{m-1}
		\left(\sum_{d_i\geq 0}q^{2m d_i}\right)
		\left(\sum_{d_m\geq 0}q^{m d_m}\right)\\
		&=
		\frac{q^{2m}}
		{(1-q^{2m})^m(1-q^m)}.
	\end{aligned}
	\]
	Summing over \(m\geq 1\), we obtain
	\[
	\sum_{\lambda\in\mathcal{P}_{\mathrm{pal}}^{\mathrm{even}}}
	q^{|\lambda|}
	=
	\sum_{m\geq 1}
	\frac{q^{2m}}
	{(1-q^{2m})^m(1-q^m)}.
	\]
	
	Finally, adjoining the empty partition, which contributes \(1\), gives
	\[
	\sum_{\lambda\in\mathcal{P}_{\mathrm{pal}}}q^{|\lambda|}
	=
	1+
	\sum_{m\geq 0}
	\frac{q^{2m+1}}{(1-q^{2m+1})^{m+1}}
	+
	\sum_{m\geq 1}
	\frac{q^{2m}}
	{(1-q^{2m})^m(1-q^m)}.
	\]
\end{proof}

We note that the generating function of palindromic-gap partitions decomposes naturally into the contributions from partitions with odd and even numbers of parts.

\begin{remark}

The generating function (\ref{ref:GF}) obtained in this paper coincides with the generating function for self-complementary ordinary partitions in their natural rectangles. However, the derivation presented here is fundamentally different.
	
The classical approach begins with self-complementary partitions inside a prescribed rectangle and derives the corresponding generating function. In contrast, our approach starts with palindromic-gap partitions and derives the generating function directly from the palindromic-gap condition. The associated rectangle is determined naturally by the partition itself, rather than being fixed in advance. Specifically, if
\[
\lambda=(\lambda_1,\lambda_2,\ldots,\lambda_\ell)
\]
is a palindromic-gap partition, then its associated rectangle is the
\[
\ell \times (\lambda_1+\lambda_\ell)
\]
rectangle, where $\ell$ is the number of parts of $\lambda$. In contrast to the classical rectangle-centered approach, our construction is partition-centered. The ambient rectangle is not prescribed a priori; instead, it is uniquely determined by the partition itself. Consequently, the generating function enumerates self-complementary partitions relative to their natural rectangles.

\end{remark}

\begin{cor}
	Let $\mathcal{P}_{\mathrm{pal}}$ denote the set of palindromic-gap partitions. Then the bivariate generating function in which $q$ records the weight and $z$ records the number of parts is
	\[
	\sum_{\lambda\in\mathcal{P}_{\mathrm{pal}}}
	z^{\ell(\lambda)}q^{|\lambda|}
	=
	1+
	\sum_{m\geq0}
	\frac{z^{2m+1}q^{2m+1}}
	{(1-q^{2m+1})^{m+1}}
	+
	\sum_{m\geq1}
	\frac{z^{2m}q^{2m}}
	{(1-q^{2m})^{m}(1-q^{m})}.
	\]
\end{cor}

\begin{proof}
	The result follows immediately from Theorem~\ref{GF}. Indeed, each partition with exactly $2m+1$ parts contributes an additional factor of $z^{2m+1}$, while each partition with exactly $2m$ parts contributes a factor of $z^{2m}$. Summing over all possible values of $m$ yields the desired generating function.
\end{proof}

\begin{cor}\label{cor:largestpart}
	Let $\mathcal{P}_{\mathrm{pal}}$ denote the set of palindromic-gap
	partitions. Then the bivariate generating function in which $q$ records
	the weight and $w$ records the largest part $\lambda_\ell$ is

 \begin{align*}
	\sum_{\lambda\in\mathcal{P}_{\mathrm{pal}}}
	w^{\lambda_\ell}q^{|\lambda|}
	&=
	1+
	\sum_{m\geq0}
	\frac{wq^{2m+1}}
	{\bigl(1-wq^{2m+1}\bigr)\bigl(1-w^2q^{2m+1}\bigr)^{m}} \\
	&+ 
	\sum_{m\geq1}
	\frac{wq^{2m}}
	{\bigl(1-wq^{2m}\bigr)\bigl(1-w^2q^{2m}\bigr)^{m-1}\bigl(1-wq^{m}\bigr)}.
\end{align*}
\end{cor}

\begin{proof}
	We follow the notation of the proof of Theorem~\ref{GF}. For
	$\lambda=(\lambda_1,\ldots,\lambda_\ell)\in\mathcal{P}_{\mathrm{pal}}$ with
	$a=\lambda_1$ and successive gaps $d_i=\lambda_{i+1}-\lambda_i$, the
	largest part is
	\[
	\lambda_\ell = a+\sum_{i=1}^{\ell-1}d_i.
	\]
	
	If $\ell=2m+1$, the independent gaps are $d_1,\ldots,d_m\geq0$, and since
	$d_{2m+1-i}=d_i$ for $1\leq i\leq m$, we have
	\[
	\lambda_\ell
	=
	a+2\bigl(d_1+\cdots+d_m\bigr).
	\]
	Together with $|\lambda|=(2m+1)(a+d_1+\cdots+d_m)$ from the proof of
	Theorem~\ref{GF}, this gives
	\[
	\sum_{\substack{\lambda\in\mathcal{P}_{\mathrm{pal}}\\ \ell(\lambda)=2m+1}}
	w^{\lambda_\ell}q^{|\lambda|}
	=
	\left(\sum_{a\geq1}w^{a}q^{(2m+1)a}\right)
	\prod_{i=1}^{m}
	\left(\sum_{d_i\geq0}w^{2d_i}q^{(2m+1)d_i}\right)
	=
	\frac{wq^{2m+1}}
	{\bigl(1-wq^{2m+1}\bigr)\bigl(1-w^2q^{2m+1}\bigr)^{m}}.
	\]
	
	If $\ell=2m$, the independent gaps are $d_1,\ldots,d_{m-1}$ together with
	the central gap $d_m$, and
	\[
	\lambda_\ell
	=
	a+2\bigl(d_1+\cdots+d_{m-1}\bigr)+d_m.
	\]
	Together with
	$|\lambda|=2m(a+d_1+\cdots+d_{m-1})+md_m$, this gives
	
\begin{align*}
	\sum_{\substack{\lambda\in\mathcal{P}_{\mathrm{pal}}\\ \ell(\lambda)=2m}}
	w^{\lambda_\ell}q^{|\lambda|}
	&=
	\left(\sum_{a\geq1}w^{a}q^{2ma}\right)
	\prod_{i=1}^{m-1}
	\left(\sum_{d_i\geq0}w^{2d_i}q^{2md_i}\right)
	\left(\sum_{d_m\geq0}w^{d_m}q^{md_m}\right)\\
	&=
	\frac{wq^{2m}}
	{\bigl(1-wq^{2m}\bigr)\bigl(1-w^2q^{2m}\bigr)^{m-1}\bigl(1-wq^{m}\bigr)}.
\end{align*}
	
	Summing over $m\geq0$ (odd case) and $m\geq1$ (even case), and adjoining
	the empty partition, which contributes $1$, yields the stated formula.
\end{proof}

\section{Enumeration by the Number of Parts} \label{enumeration}

The generating functions derived in the previous section immediately yield explicit closed formulas for the number of palindromic-gap partitions with a prescribed number of parts. In this section, we derive these formulas separately for partitions with odd and even numbers of parts. 

Let $p_{\mathrm{pg}}(n,r)$ denote the number of palindromic-gap partitions of \(n\) into exactly \(r\) parts.

\begin{cor} \label{oddnumberofparts}
	For every integer \(m\geq0\),
	\[
	\sum_{n\geq0}p_{\mathrm{pg}}(n,2m+1)q^n
	=
	\frac{q^{2m+1}}{(1-q^{2m+1})^{m+1}}.
	\]
	Consequently,
	\[
	p_{\mathrm{pg}}(n,2m+1)=
	\begin{cases}
		\displaystyle
		\binom{\frac{n}{2m+1}+m-1}{m},
		&
		2m+1\mid n,\\[3mm]
		0,
		&
		\text{otherwise}.
	\end{cases}
	\]
\end{cor}

\begin{proof}
	Using the binomial expansion
	\[
	\frac{1}{(1-x)^{m+1}}
	=
	\sum_{j\geq0}
	\binom{j+m}{m}x^j,
	\]
	with \(x=q^{2m+1}\), we obtain
	\[
	\frac{q^{2m+1}}{(1-q^{2m+1})^{m+1}}
	=
	\sum_{j\geq0}
	\binom{j+m}{m}
	q^{(2m+1)(j+1)}.
	\]
	Hence the coefficient of \(q^n\) is nonzero only when
	\(n=(2m+1)s\) for some integer \(s\geq1\). Since
	\(s=j+1\), we obtain
	\[
	p_{\mathrm{pg}}(n,2m+1)
	=
	\binom{s+m-1}{m},
	\]
	which is precisely the stated formula.
\end{proof}

\begin{ex}
	Consider palindromic-gap partitions of \(20\) into exactly \(5\) parts.
	Here \(2m+1=5\), so \(m=2\). Since \(5\mid 20\), the odd-length
	coefficient formula gives
	\[
	p_{\mathrm{pg}}(20,5)
	=
	\binom{\frac{20}{5}+2-1}{2}
	=
	\binom{5}{2}
	=
	10.
	\]
	
	Indeed, a palindromic-gap partition with five parts can be written as
	\[
	\lambda
	=
	a,\,
	a+d_1,\,
	a+d_1+d_2,\,
	a+d_1+2d_2,\,
	a+2d_1+2d_2,
	\]
	where \(a\geq 1\) and \(d_1,d_2\geq0\). Its weight is
	\[
	|\lambda|=5(a+d_1+d_2).
	\]
	Thus, the condition \(|\lambda|=20\) is equivalent to
	\[
	a+d_1+d_2=4.
	\]
	Equivalently,
	\[
	(a-1)+d_1+d_2=3,
	\]
	which has
	\[
	\binom{3+3-1}{3-1}=\binom{5}{2}=10
	\]
	nonnegative integer solutions. The corresponding partitions are
	\[
	\begin{aligned}
		&(1,1,4,7,7), &&(1,2,4,6,7),\\
		&(1,3,4,5,7), &&(1,4,4,4,7),\\
		&(2,2,4,6,6), &&(2,3,4,5,6),\\
		&(2,4,4,4,6), &&(3,3,4,5,5),\\
		&(3,4,4,4,5), &&(4,4,4,4,4).
	\end{aligned}
	\]
	Each of these partitions has a palindromic gap sequence.
\end{ex}

\begin{cor} \label{evennumberofparts}
	For every integer \(m\geq1\),
	\[
	\sum_{n\geq0}p_{\mathrm{pg}}(n,2m)q^n
	=
	\frac{q^{2m}}
	{(1-q^{2m})^m(1-q^m)}.
	\]
	Consequently,
	\[
	p_{\mathrm{pg}}(n,2m)
	=
	\begin{cases}
		\displaystyle
		\binom{\left\lfloor\frac{n}{2m}\right\rfloor+m-1}{m},
		&
		m\mid n,\; n\geq2m,\\[3mm]
		0,
		&
		\text{otherwise}.
	\end{cases}
	\]
\end{cor}

\begin{proof}
	Let \(x=q^m\). Then
	\[
	\frac{q^{2m}}
	{(1-q^{2m})^m(1-q^m)}
	=
	\frac{x^2}{(1-x^2)^m(1-x)}.
	\]
	Using
	\[
	\frac{1}{(1-x^2)^m}
	=
	\sum_{j\geq0}
	\binom{j+m-1}{m-1}x^{2j},
	\]
	and
	\[
	\frac{1}{1-x}
	=
	\sum_{k\geq0}x^k,
	\]
	we obtain
	\[
	\frac{x^2}{(1-x^2)^m(1-x)}
	=
	\sum_{j,k\geq0}
	\binom{j+m-1}{m-1}
	x^{2+2j+k}.
	\]
	Therefore, if \(n=mN\), then
	\[
	p_{\mathrm{pg}}(mN,2m)
	=
	\sum_{j=0}^{\lfloor(N-2)/2\rfloor}
	\binom{j+m-1}{m-1}.
	\]
	Applying the hockey-stick identity
	\[
	\sum_{j=0}^{L}
	\binom{j+m-1}{m-1}
	=
	\binom{L+m}{m},
	\]
	with
	\[
	L=\left\lfloor\frac{N-2}{2}\right\rfloor,
	\]
	yields
	\[
	p_{\mathrm{pg}}(mN,2m)
	=
	\binom{\left\lfloor\frac{N}{2}\right\rfloor+m-1}{m}.
	\]
	Since
	\[
	\left\lfloor\frac{N}{2}\right\rfloor
	=
	\left\lfloor\frac{n}{2m}\right\rfloor,
	\]
	the desired formula follows.
\end{proof}

\begin{ex}
	Consider palindromic-gap partitions of \(12\) into exactly \(4\) parts.
	Here \(2m=4\), so \(m=2\). Since \(2\mid 12\), the even-length
	coefficient formula gives
	\[
	p_{\mathrm{pg}}(12,4)
	=
	\binom{\left\lfloor\frac{12}{4}\right\rfloor+2-1}{2}
	=
	\binom{4}{2}
	=
	6.
	\]
	
	A palindromic-gap partition with four parts has the form
	\[
	\lambda
	=
	(a,\,
	a+d_1,\,
	a+d_1+d_2,\,
	a+2d_1+d_2),
	\]
	where \(a\geq1\) and \(d_1,d_2\geq0\). Its weight is
	\[
	|\lambda|
	=
	4a+4d_1+2d_2.
	\]
	Therefore, the condition \(|\lambda|=12\) is equivalent to
	\[
	2a+2d_1+d_2=6.
	\]
	The nonnegative solutions, subject to \(a\geq1\), yield the six
	partitions
	\[
	\begin{aligned}
		&(1,1,5,5), &&(1,2,4,5), &&(1,3,3,5),\\
		&(2,2,4,4), &&(2,3,3,4), &&(3,3,3,3).
	\end{aligned}
	\]
	Their respective gap sequences are
	\[
	(0,4,0),\quad
	(1,2,1),\quad
	(2,0,2),\quad
	(0,2,0),\quad
	(1,0,1),\quad
	(0,0,0),
	\]
	and hence all six are palindromic-gap partitions.
\end{ex}
\section{Bijections with Ordinary Partitions} \label{bijections}

In this section, we establish explicit bijections between palindromic-gap partitions and ordinary partitions. The constructions differ according to whether the number of parts is odd or even, and are therefore presented separately. These bijections reveal an unexpected combinatorial structure underlying palindromic-gap partitions and provide a direct combinatorial interpretation of the generating functions obtained in the previous section. We emphasize that these bijections are not weight-preserving in the usual sense. Instead, the weight of a palindromic-gap partition is encoded by the largest part, or the two largest parts, of its image, depending on the parity of the number of parts.

Throughout this section, if an ordinary partition has fewer than two
parts, the missing largest parts are taken to be zero. In particular, the largest part of the empty partition is taken to be zero.

\begin{theorem}\label{thm:odd_bijection}
	For every integer $m\geq 0$, there exists an explicit bijection
	\[
	\Phi_{2m+1}:\mathcal{P}_{\mathrm{pal}}^{\,2m+1}
	\longrightarrow
	\mathcal{P}_{\leq m+1},
	\]
	where $\mathcal{P}_{\mathrm{pal}}^{\,2m+1}$ denotes the set of
	palindromic-gap partitions with exactly $2m+1$ parts, and
	$\mathcal{P}_{\leq m+1}$ denotes the set of ordinary partitions with at
	most $m+1$ parts.
	
	Moreover, if
	\[
	\Phi_{2m+1}(\lambda)=\mu
	\]
	and
	\[
	\mu=(\mu_1,\mu_2,\ldots,\mu_{\ell(\mu)})
	\]
	is written in non-decreasing order, then
	\[
	|\lambda|
	=
	(2m+1)\bigl(\mu_{\ell(\mu)}+1\bigr).
	\]
\end{theorem}

\begin{proof}
	Let
	\[
	\lambda=(\lambda_1,\lambda_2,\ldots,\lambda_{2m+1})
	\in\mathcal{P}_{\mathrm{pal}}^{\,2m+1},
	\]
	where the parts are written in non-decreasing order. Let
	\[
	s=\lambda_1
	\]
	denote the smallest part, and let
	\[
	d_i=\lambda_{i+1}-\lambda_i,\qquad 1\le i\le m.
	\]
	Since $\lambda$ is a palindromic-gap partition, its gap sequence is
	\[
	(d_1,d_2,\ldots,d_m,d_m,\ldots,d_2,d_1).
	\]
	
	Define
	\[
	x_0=s-1,\qquad
	x_i=d_i,\quad 1\le i\le m,
	\]
	and let
	\[
	a_j=x_0+x_1+\cdots+x_{j-1},
	\qquad
	1\le j\le m,
	\]
	together with
	\[
	R=x_0+x_1+\cdots+x_m.
	\]
	
We define
\begin{center}
	$\Phi_{2m+1}(\lambda)=(a_1,a_2,\dots,a_m,R)$,
\end{center}
 where $0\le a_1\le\cdots\le a_m\le R$. Note that the leading entries of this sequence can be zero (this occurs exactly when $s=1$ together with vanishing initial gaps $d_1,d_2,\dots$). Since a partition cannot have $0$ as a part, we delete these leading zeros before regarding $(a_1,\dots,a_m,R)$ as a partition --- thereby obtaining $\Phi_{2m+1}(\lambda)$, an ordinary partition with at most $m+1$ parts.
	
	Conversely, let
	\[
	\mu\in\mathcal{P}_{\le m+1}.
	\]
	By adjoining leading zeros if necessary, write
	\[
	\mu=(b_1,b_2,\ldots,b_{m+1}),
	\]
	where
	\[
	0\le b_1\le b_2\le\cdots\le b_{m+1}.
	\]
	Define
	\[
	x_0=b_1,
	\qquad
	x_i=b_{i+1}-b_i,
	\quad
	1\le i\le m.
	\]
	Then set
	\[
	s=x_0+1,
	\qquad
	d_i=x_i,
	\quad
	1\le i\le m.
	\]
	These data uniquely determine a partition whose gap sequence is
	\[
	(d_1,d_2,\ldots,d_m,d_m,\ldots,d_2,d_1),
	\]
	so the inverse map is well defined. Hence $\Phi_{2m+1}$ is a bijection.
	
	Finally,
	\[
	|\lambda|
	=
	(2m+1)(s+d_1+\cdots+d_m).
	\]
	Since
	\[
	s+d_1+\cdots+d_m
	=
	1+x_0+\cdots+x_m
	=
	R+1,
	\]
	we obtain
	\[
	|\lambda|
	=
	(2m+1)(R+1).
	\]
	Because $R$ is the largest part of $\mu$, it follows that
	\[
	|\lambda|
	=
	(2m+1)\bigl(\mu_{\ell(\mu)}+1\bigr).
	\]
The largest part of the empty partition is taken to be zero; this convention is only needed when
\[
\lambda=(1,1,\ldots,1),
\]
in which case
\[
\mu=\Phi_{2m+1}(\lambda)
\]
is the empty partition.
	This completes the proof.
\end{proof}

\begin{ex}\label{ex:odd_bijection}
	Consider the palindromic-gap partition
	\[
	\lambda=2,3,5,7,8
	\in\mathcal{P}_{\mathrm{pal}}^{\,5}.
	\]
	Here $2m+1=5$, so $m=2$. The smallest part is
	\[
	s=2,
	\]
	and the gap sequence is
	\[
	(1,2,2,1).
	\]
	Thus,
	\[
	d_1=1
	\qquad\text{and}\qquad
	d_2=2.
	\]
	According to the definition of $\Phi_5$, we set
	\[
	x_0=s-1=1,\qquad x_1=d_1=1,\qquad x_2=d_2=2.
	\]
	The corresponding partial sums are
	\[
	a_1=x_0=1,
	\qquad
	a_2=x_0+x_1=2,
	\]
	and
	\[
	R=x_0+x_1+x_2=4.
	\]
	Therefore,
	\[
	\Phi_5(\lambda)=1,2,4.
	\]
	
	We now apply the inverse map to
	\[
	\mu=1,2,4.
	\]
	Since $\mu$ already has $m+1=3$ parts, no leading zeros are required.
	Writing
	\[
	(b_1,b_2,b_3)=(1,2,4),
	\]
	we recover
	\[
	x_0=b_1=1,
	\]
	\[
	x_1=b_2-b_1=1,
	\qquad
	x_2=b_3-b_2=2.
	\]
	Hence
	\[
	s=x_0+1=2,
	\qquad
	d_1=x_1=1,
	\qquad
	d_2=x_2=2.
	\]
	The resulting gap sequence is
	\[
	(1,2,2,1),
	\]
	which, together with the smallest part $s=2$, reconstructs
	\[
	\lambda=2,3,5,7,8.
	\]
Thus, these two constructions are inverses of each other.
	
	Moreover,
	\[
	|\lambda|=2+3+5+7+8=25,
	\]
	while
	\[
	(2m+1)\bigl(\mu_{\ell(\mu)}+1\bigr)
	=
	5(4+1)=25,
	\]
	in agreement with Theorem~\ref{thm:odd_bijection}.
\end{ex}

\begin{theorem}\label{thm:even_bijection}
	For every integer $m\geq 1$, there exists an explicit bijection
	\[
	\Phi_{2m}:\mathcal{P}_{\mathrm{pal}}^{\,2m}
	\longrightarrow
	\mathcal{P}_{\leq m+1},
	\]
	where $\mathcal{P}_{\mathrm{pal}}^{\,2m}$ denotes the set of
	palindromic-gap partitions with exactly $2m$ parts, and
	$\mathcal{P}_{\leq m+1}$ denotes the set of ordinary partitions with at
	most $m+1$ parts.
	
	Moreover, if
	\[
	\Phi_{2m}(\lambda)=\mu
	\]
	and
	\[
	\mu=(\mu_1,\mu_2,\ldots,\mu_{\ell(\mu)})
	\]
	is written in non-decreasing order, then
	\[
	|\lambda|
	=
	m\bigl(\mu_{\ell(\mu)-1}+\mu_{\ell(\mu)}+2\bigr),
	\]
	where the second largest part is taken to be zero when $\mu$ has only
	one part.
\end{theorem}

\begin{proof}
	Let
	\[
	\lambda=(\lambda_1,\lambda_2,\ldots,\lambda_{2m})
	\in\mathcal{P}_{\mathrm{pal}}^{\,2m},
	\]
	where the parts are written in non-decreasing order. Let
	\[
	s=\lambda_1.
	\]
	The palindromic gap sequence of $\lambda$ has the form
	\[
	(d_1,d_2,\ldots,d_{m-1},e,d_{m-1},\ldots,d_2,d_1),
	\]
	where
	\[
	d_i=\lambda_{i+1}-\lambda_i,
	\qquad
	1\le i\le m-1,
	\]
	and
	\[
	e=\lambda_{m+1}-\lambda_m.
	\]
	
	Define
	\[
	x_0=s-1,\qquad
	x_i=d_i,\quad 1\le i\le m-1,\qquad
	x_m=e.
	\]
	For
	\[
	1\le j\le m,
	\]
	let
	\[
	a_j=x_0+x_1+\cdots+x_{j-1},
	\]
	and define
	\[
	R=x_0+x_1+\cdots+x_m.
	\]
	
We define
\begin{center}
	$\Phi_{2m}(\lambda)=(a_1,a_2,\dots,a_m,R)$,
\end{center}
where $0\le a_1\le\cdots\le a_m\le R$. Note that the leading entries of this sequence can be zero (this occurs exactly when $s=1$ together with vanishing initial gaps $d_1,d_2,\dots$). Since a partition cannot have $0$ as a part, we delete these leading zeros before regarding $(a_1,\dots,a_m,R)$ as a partition --- thereby obtaining $\Phi_{2m}(\lambda)$, an ordinary partition with at most $m+1$ parts.

	Conversely, let
	\[
	\mu\in\mathcal{P}_{\le m+1}.
	\]
	By adjoining leading zeros if necessary, write
	\[
	\mu=(b_1,b_2,\ldots,b_{m+1}),
	\]
	with
	\[
	0\le b_1\le b_2\le\cdots\le b_{m+1}.
	\]
	Define
	\[
	x_0=b_1,
	\qquad
	x_i=b_{i+1}-b_i,
	\quad
	1\le i\le m.
	\]
	Then let
	\[
	s=x_0+1,\qquad
	d_i=x_i,\quad 1\le i\le m-1,\qquad
	e=x_m.
	\]
	These values uniquely reconstruct a partition whose gap sequence is
	\[
	(d_1,d_2,\ldots,d_{m-1},e,d_{m-1},\ldots,d_2,d_1).
	\]
	Thus $\Phi_{2m}$ is a bijection.
	
	To determine the weight,
	\[
	|\lambda|
	=
	2m(s+d_1+\cdots+d_{m-1})+me.
	\]
	Since
	\[
	a_m
	=
	x_0+x_1+\cdots+x_{m-1}
	=
	s+d_1+\cdots+d_{m-1}-1,
	\]
	and
	\[
	R=a_m+e,
	\]
	we obtain
	\begin{align*}
		|\lambda|
		&=
		m\bigl(2a_m+e+2\bigr)\\
		&=
		m(a_m+R+2).
	\end{align*}
	Since $a_m$ and $R$ are respectively the second largest and the largest parts of $\mu$, we conclude that
	\[
	|\lambda|
	=
	m\bigl(\mu_{\ell(\mu)-1}+\mu_{\ell(\mu)}+2\bigr).
	\]
	If $\mu$ is the empty partition (which occurs precisely when
	$\lambda=(1,1,\dots,1)$), both $\mu_{\ell(\mu)-1}$ and $\mu_{\ell(\mu)}$ are
	taken to be zero, consistent with the convention in the statement of
	Theorem \ref{thm:even_bijection}.
	This completes the proof.
\end{proof}

\begin{ex}\label{ex:even_bijection}
	Consider the palindromic-gap partition
	\[
	\lambda=2,3,5,8,10,11
	\in\mathcal{P}_{\mathrm{pal}}^{\,6}.
	\]
	Here $2m=6$, so $m=3$. The smallest part is
	\[
	s=2,
	\]
	and the gap sequence is
	\[
	(1,2,3,2,1).
	\]
	Thus,
	\[
	d_1=1,\qquad d_2=2,\qquad e=3.
	\]
	According to the definition of $\Phi_6$, we set
	\[
	x_0=s-1=1,\qquad
	x_1=d_1=1,\qquad
	x_2=d_2=2,\qquad
	x_3=e=3.
	\]
	The corresponding partial sums are
	\[
	a_1=x_0=1,
	\]
	\[
	a_2=x_0+x_1=2,
	\]
	and
	\[
	a_3=x_0+x_1+x_2=4.
	\]
	Furthermore,
	\[
	R=x_0+x_1+x_2+x_3=7.
	\]
	Therefore,
	\[
	\Phi_6(\lambda)=1,2,4,7.
	\]
	
	We now apply the inverse map to
	\[
	\mu=1,2,4,7.
	\]
	Since $\mu$ already has $m+1=4$ parts, no leading zeros are required.
	Writing
	\[
	(b_1,b_2,b_3,b_4)=(1,2,4,7),
	\]
	we obtain
	\[
	x_0=b_1=1,
	\]
	\[
	x_1=b_2-b_1=1,
	\qquad
	x_2=b_3-b_2=2,
	\qquad
	x_3=b_4-b_3=3.
	\]
	Consequently,
	\[
	s=x_0+1=2,
	\qquad
	d_1=x_1=1,
	\qquad
	d_2=x_2=2,
	\qquad
	e=x_3=3.
	\]
	The resulting gap sequence is
	\[
	(1,2,3,2,1),
	\]
	which, together with the smallest part $s=2$, reconstructs
	\[
	\lambda=2,3,5,8,10,11.
	\]
Thus, these two constructions are inverses of each other.
	
	Finally,
	\[
	|\lambda|
	=
	2+3+5+8+10+11
	=
	39,
	\]
	whereas
	\[
	m\bigl(
	\mu_{\ell(\mu)-1}
	+
	\mu_{\ell(\mu)}
	+2
	\bigr)
	=
	3(4+7+2)
	=
	39,
	\]
	as asserted in Theorem~\ref{thm:even_bijection}.
\end{ex}

\begin{cor}
	For every $m \ge 1$, there is an explicit bijection
	\[
	\Psi_m := \Phi_{2m}^{-1} \circ \Phi_{2m+1} \;:\; \mathcal{P}_{\mathrm{pal}}^{2m+1} \longrightarrow \mathcal{P}_{\mathrm{pal}}^{2m}.
	\]
	Explicitly, if $\lambda \in \mathcal{P}_{\mathrm{pal}}^{2m+1}$ has smallest part $s$ and gap sequence
	\[
	(d_1,\dots,d_m,d_m,\dots,d_1),
	\]
	then $\Psi_m(\lambda)$ is the partition in $\mathcal{P}_{\mathrm{pal}}^{2m}$ with smallest part $s$ and gap sequence
	\[
	(d_1,\dots,d_{m-1},d_m,d_{m-1},\dots,d_1).
	\]
	That is, $\Psi_m$ simply reinterprets the last gap $d_m$ of $\lambda$ as the middle gap $e$ of the shorter palindromic-gap partition, leaving $s,d_1,\dots,d_{m-1}$ unchanged. In particular, writing $T = s+d_1+\cdots+d_m$,
	\[
	|\lambda| = (2m+1)T, \qquad |\Psi_m(\lambda)| = m(2T - d_m).
	\]
\end{cor}

\begin{proof}
	Let $\lambda \in \mathcal{P}_{\mathrm{pal}}^{2m+1}$, with smallest part $s$ and gaps $d_1,\dots,d_m$ as in the proof of Theorem \ref{thm:odd_bijection}. Set $x_0=s-1$, $x_i=d_i$ for $1\le i\le m$, and
	\[
	a_j = x_0+x_1+\cdots+x_{j-1} \ (1\le j\le m), \qquad R = x_0+x_1+\cdots+x_m,
	\]
	so that $\Phi_{2m+1}(\lambda) = (a_1,\dots,a_m,R)$, with $0\le a_1\le\cdots\le a_m\le R$, after deleting leading zeros.
	
	To apply $\Phi_{2m}^{-1}$, we regard $\Phi_{2m+1}(\lambda)$ as an element of $\mathcal{P}_{\le m+1}$ and, following the proof of Theorem \ref{thm:even_bijection}, restore it to a sequence of exactly $m+1$ terms by adjoining leading zeros if necessary. Since $(a_1,\dots,a_m,R)$ already has exactly $m+1$ terms and satisfies $0\le a_1\le\cdots\le a_m\le R$, no adjustment is needed: we simply set
	\[
	b_i = a_i \ (1\le i\le m), \qquad b_{m+1} = R.
	\]
	
	Now apply the inverse construction of Theorem \ref{thm:even_bijection}: set
	\[
	x_0' = b_1, \qquad x_i' = b_{i+1}-b_i \ (1\le i \le m),
	\]
	and then
	\[
	s' = x_0'+1, \qquad d_i' = x_i' \ (1\le i\le m-1), \qquad e' = x_m'.
	\]
	
	Since $b_i = a_i = x_0+x_1+\cdots+x_{i-1}$ for $1\le i\le m$ and $b_{m+1}=R = x_0+\cdots+x_m$, we get
	\[
	x_0' = b_1 = a_1 = x_0 = s-1 \implies s' = s,
	\]
	\[
	x_i' = b_{i+1}-b_i = a_{i+1}-a_i = x_i = d_i \ (1\le i\le m-1) \implies d_i' = d_i,
	\]
	\[
	x_m' = b_{m+1}-b_m = R - a_m = x_m = d_m \implies e' = d_m.
	\]
	
	Hence $\Psi_m(\lambda)$ is exactly the partition in $\mathcal{P}_{\mathrm{pal}}^{2m}$ with smallest part $s$ and the gap sequence
	$(d_1,\dots,d_{m-1},d_m,d_{m-1},\dots,d_1)$, as claimed.
	
	For the weight, Theorem \ref{thm:odd_bijection} gives $|\lambda| = (2m+1)(s+d_1+\cdots+d_m) = (2m+1)T$. Applying the weight formula of Theorem \ref{thm:even_bijection} to $\Psi_m(\lambda)$, with smallest part $s'=s$ and gaps $d_1',\dots,d_{m-1}',e' = d_1,\dots,d_{m-1},d_m$, we obtain
	\begin{align*}
		|\Psi_m(\lambda)| &= 2m\bigl(s'+d_1'+\cdots+d_{m-1}'\bigr)+m\,e' \\
		&= 2m\bigl(s+d_1+\cdots+d_{m-1}\bigr)+m\,d_m \\
		&= 2m(T-d_m)+m\,d_m \\
		&= m(2T-d_m).
	\end{align*}

	This completes the proof.
\end{proof}

Having established explicit bijections with ordinary partitions, we now turn to a geometric characterization of palindromic-gap partitions in terms of their Ferrers diagrams.

\section{A Characterization via Ferrers Diagrams} \label{ferrers diagrams}
In this section, we present a geometric characterization of palindromic-gap partitions in terms of their Ferrers diagrams. We show that the defining palindromic-gap condition is equivalent to the Ferrers diagram being self-complementary under a $180^\circ$ rotation inside a suitable rectangle.

We recall that a Ferrers diagram is said to be \emph{self-complementary} inside a rectangle if, after taking its complement in the rectangle and rotating it by $180^\circ$, one recovers the original Ferrers diagram.

\begin{theorem}\label{thm:ferrers}
Let
\[
\lambda=(\lambda_1,\lambda_2,\ldots,\lambda_\ell)
\]
be a partition written in non-decreasing order. Then the following are equivalent:

(i) $\lambda$ is a palindromic-gap partition.

(ii)
\[
\lambda_i+\lambda_{\ell+1-i}
=
\lambda_1+\lambda_\ell,
\qquad
1\le i\le\ell.
\]

(iii) When Ferrers diagrams are drawn with left-justified rows, the Ferrers diagram
of $\lambda$ is self-complementary under a $180^\circ$ rotation inside the rectangle
of dimensions
\[
\ell\times(\lambda_1+\lambda_\ell).
\]
\end{theorem}

\begin{proof}
	For $1\leq i\leq \ell-1$, define the successive gaps of $\lambda$ by
	\[
	d_i=\lambda_{i+1}-\lambda_i.
	\]
	Since the parts of $\lambda$ are written in non-decreasing order, we have
	$d_i\geq 0$.
	
	Suppose first that $\lambda$ is a palindromic-gap partition. Then
	\[
	d_i=d_{\ell-i},
	\qquad 1\leq i\leq \ell-1.
	\]
	For each $1\leq i\leq \ell$, we have
	\[
	\lambda_i-\lambda_1
	=\sum_{j=1}^{i-1}d_j
	\]
	and
	\[
	\lambda_\ell-\lambda_{\ell+1-i}
	=\sum_{j=\ell+1-i}^{\ell-1}d_j.
	\]
	Using the palindromicity of the gap sequence, we obtain
	\[
	\sum_{j=\ell+1-i}^{\ell-1}d_j
	=\sum_{j=1}^{i-1}d_j.
	\]
	Therefore,
	\[
	\lambda_i-\lambda_1
	=\lambda_\ell-\lambda_{\ell+1-i},
	\]
	and hence
	\[
	\lambda_i+\lambda_{\ell+1-i}
	=\lambda_1+\lambda_\ell
	\]
	for every $1\leq i\leq\ell$.
	
	Conversely, suppose that
	\[
	\lambda_i+\lambda_{\ell+1-i}
	=\lambda_1+\lambda_\ell,
	\qquad 1\leq i\leq\ell.
	\]
	Applying this identity with indices $i$ and $i+1$, and subtracting the two
	equalities, gives
	\[
	\lambda_{i+1}-\lambda_i
	=\lambda_{\ell+1-i}-\lambda_{\ell-i},
	\qquad 1\leq i\leq\ell-1.
	\]
	Thus
	\[
	d_i=d_{\ell-i},
	\qquad 1\leq i\leq\ell-1,
	\]
	so the successive gap sequence of $\lambda$ is palindromic. Hence $\lambda$
	is a palindromic-gap partition.
	
	It remains to verify the Ferrers-diagram interpretation. Set
	\[
	C=\lambda_1+\lambda_\ell.
	\]
	Consider the Ferrers diagram of $\lambda$ inside the rectangle of dimensions
	\[
	\ell\times C.
	\]
	The complement of the $i$th row has length $C-\lambda_i$. After a
	$180^\circ$ rotation, this row occupies the position of the
	$(\ell+1-i)$th row. Consequently, the Ferrers diagram is
	self-complementary under a $180^\circ$ rotation if and only if
	\[
	C-\lambda_i=\lambda_{\ell+1-i},
	\qquad 1\leq i\leq\ell.
	\]
	This is equivalent to
	\[
	\lambda_i+\lambda_{\ell+1-i}
	=C
	=\lambda_1+\lambda_\ell,
	\qquad 1\leq i\leq\ell.
	\]
	Therefore, all three conditions are equivalent.
\end{proof}

\begin{ex}
Consider the partition
\[
\lambda=1,3,7,11,13.
\]
Its successive gap sequence is
\[
(2,4,4,2),
\]
which is palindromic. Furthermore,
\[
\lambda_1+\lambda_5
=
\lambda_2+\lambda_4
=
2\lambda_3
=
14.
\]
Equivalently,
\[
\lambda_i+\lambda_{6-i}
=
\lambda_1+\lambda_5
=
14,
\qquad 1\leq i\leq 5.
\]
Hence, by Theorem~\ref{thm:ferrers}, the Ferrers diagram of $\lambda$ is
self-complementary under a $180^\circ$ rotation inside the rectangle
of dimensions
\[
5\times(\lambda_1+\lambda_5)=5\times14.
\]
	This is illustrated below.
	
	\medskip
	
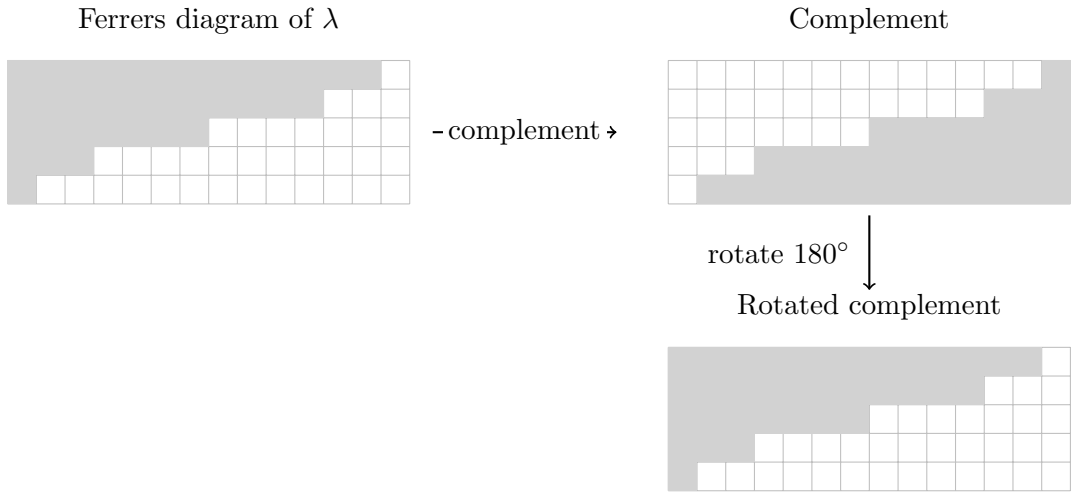
\begin{center}
	\begin{tikzpicture}[scale=0.38]
		
		
		\begin{scope}
			\node at (7,6.4)
			{\small Ferrers diagram of $\lambda$};
			
			\draw[step=1,gray!50] (0,0) grid (14,5);
			
			\foreach \x in {0}
			\fill[gray!35] (\x,0) rectangle ++(1,1);
			
			\foreach \x in {0,...,2}
			\fill[gray!35] (\x,1) rectangle ++(1,1);
			
			\foreach \x in {0,...,6}
			\fill[gray!35] (\x,2) rectangle ++(1,1);
			
			\foreach \x in {0,...,10}
			\fill[gray!35] (\x,3) rectangle ++(1,1);
			
			\foreach \x in {0,...,12}
			\fill[gray!35] (\x,4) rectangle ++(1,1);
		\end{scope}
		
		
		\draw[->,thick]
		(14.8,2.5)--(21.2,2.5)
		node[midway,fill=white,inner sep=2pt]
		{\small complement};
		
		
		\begin{scope}[xshift=23cm]
			
			\node at (7,6.4)
			{\small Complement};
			
			\draw[step=1,gray!50] (0,0) grid (14,5);
			
			\foreach \x in {1,...,13}
			\fill[gray!35] (\x,0) rectangle ++(1,1);
			
			\foreach \x in {3,...,13}
			\fill[gray!35] (\x,1) rectangle ++(1,1);
			
			\foreach \x in {7,...,13}
			\fill[gray!35] (\x,2) rectangle ++(1,1);
			
			\foreach \x in {11,...,13}
			\fill[gray!35] (\x,3) rectangle ++(1,1);
			
			\foreach \x in {13}
			\fill[gray!35] (\x,4) rectangle ++(1,1);
		\end{scope}
		
		
		\draw[->,thick]
		(30,-0.4)--(30,-3.0);
		
		\node[left] at (29.7,-1.7)
		{\small rotate $180^\circ$};
		
		
		\begin{scope}[xshift=23cm,yshift=-10cm]
			
			\node at (7,6.4)
			{\small Rotated complement};
			
			\draw[step=1,gray!50] (0,0) grid (14,5);
			
			\foreach \x in {0}
			\fill[gray!35] (\x,0) rectangle ++(1,1);
			
			\foreach \x in {0,...,2}
			\fill[gray!35] (\x,1) rectangle ++(1,1);
			
			\foreach \x in {0,...,6}
			\fill[gray!35] (\x,2) rectangle ++(1,1);
			
			\foreach \x in {0,...,10}
			\fill[gray!35] (\x,3) rectangle ++(1,1);
			
			\foreach \x in {0,...,12}
			\fill[gray!35] (\x,4) rectangle ++(1,1);
		\end{scope}
		
	\end{tikzpicture}
\end{center}
	
\end{ex}

\begin{cor}
	Let
	\[
	\lambda=(\lambda_1,\lambda_2,\ldots,\lambda_\ell)
	\]
	be a palindromic-gap partition. Then
	\[
	|\lambda|
	=
	\frac{\ell(\lambda_1+\lambda_\ell)}{2}.
	\]
	Equivalently, the area of the rectangle of dimensions
	\[
	\ell\times(\lambda_1+\lambda_\ell)
	\]
	is equal to twice the weight of $\lambda$.
\end{cor}

\begin{proof}
	By Theorem \ref{thm:ferrers},
	\[
	\lambda_i+\lambda_{\ell+1-i}
	=
	\lambda_1+\lambda_\ell,
	\qquad
	1\le i\le\ell.
	\]
	Therefore,
	\[
	2|\lambda|
	=
	\sum_{i=1}^{\ell}
	\left(\lambda_i+\lambda_{\ell+1-i}\right)
	=
	\ell(\lambda_1+\lambda_\ell),
	\]
	which gives
	\[
	|\lambda|
	=
	\frac{\ell(\lambda_1+\lambda_\ell)}{2}.
	\]
	The second statement follows immediately since the rectangle of dimensions
	$\ell\times(\lambda_1+\lambda_\ell)$ has area
	$\ell(\lambda_1+\lambda_\ell)=2|\lambda|$.
\end{proof}

The quantity $\ell(\lambda_1+\lambda_\ell)$ appearing above is precisely the area
condition that Keith~\cite{Keith2011} uses to define a partition of $n$ as
\emph{complementable}: a partition $\lambda = (\lambda_1,\dots,\lambda_k)$ is
complementable if
\[
k(\lambda_1+\lambda_k) = 2n,
\]
this being a necessary condition for $\lambda$ to be self-complementary inside
some rectangle, since such a rectangle must have height $k$ and width
$\lambda_1+\lambda_k$. Our corollary therefore shows that every
palindromic-gap partition is automatically complementable in the sense of
Keith. However, as Keith himself remarks without giving an example,
complementability is strictly weaker than genuine self-complementarity: The first complementable partition that is not self-complementary in the associated box occurs for $n = 15$. We give an explicit example below.

\begin{ex} \label{ex:n15}
	Consider $\lambda = 1,2,2,5,5$, a partition of $n=15$ with $\ell=5$ parts.
	
	\medskip
	\noindent\textbf{Step 1: $\lambda$ is complementable.} By Keith's area
	condition,
	\[
	\ell(\lambda_1+\lambda_\ell) = 5(1+5) = 30 = 2n,
	\]
	so $\lambda$ is complementable in the sense of~\cite{Keith2011}.
	
	\medskip
	\noindent\textbf{Step 2: $\lambda$ is not self-complementary.} Self-complementarity
	in the $5\times 6$ box requires $\lambda_i + \lambda_{6-i} = 6$ for every
	$1\le i\le 5$. Checking each pair:
	\[
	\lambda_1+\lambda_5 = 1+5 = 6, \qquad
	\lambda_2+\lambda_4 = 2+5 = 7 \neq 6.
	\]
	The second identity fails, so $\lambda$ is \emph{not} self-complementary in its
	associated rectangle, even though it satisfies the area condition of Step~1.
	
	\medskip
	\noindent\textbf{Step 3: consistency with the gap sequence.} The gap sequence
	of $\lambda$ is
	\[
	g = (\lambda_2-\lambda_1,\ \lambda_3-\lambda_2,\ \lambda_4-\lambda_3,\ \lambda_5-\lambda_4) = (1,0,3,0),
	\]
	which is not palindromic ($d_1=1\neq d_4=0$), consistent with Step~2 via
	Theorem \ref{thm:ferrers}.
	
	\medskip
	\noindent Thus $\lambda$ is complementable but not self-complementary.
\end{ex}

\section{Conclusion}
In the closing section of his paper, Keith poses the following question: how
many complementable and self-complementary partitions of $n$ are there, for
both ordinary and signed partitions, and can one give explicit generating
functions for these two families that also describe how they differ? Our
results give a partial answer to this question. Theorem \ref{thm:ferrers} identifies the
self-complementary partitions of $n$ with the palindromic-gap partitions of
$n$, Theorem~\ref{GF} gives an explicit generating
function for the latter, and
Corollaries~\ref{oddnumberofparts}--\ref{evennumberofparts} refine this by
number of parts. We do not address the signed case, nor do we give a generating function for complementable partitions that are \emph{not} self-complementary --- Example~\ref{ex:n15} above shows only that the two notions diverge, without enumerating the differencen. These remain open, as originally posed by Keith.

\bibliographystyle{amsplain}

\end{document}